\documentclass{amsart}

\usepackage{latexsym,enumerate}
\usepackage{amsmath,amsthm,amsopn,amstext,amscd,amsfonts,amssymb}
\usepackage[ansinew]{inputenc}
\usepackage{verbatim}
\usepackage{wasysym}
\usepackage[all]{xy}
\usepackage{epsfig} 
\usepackage{graphicx,psfrag} 
\usepackage{subfigure}
\usepackage{pstricks,pst-node}

\usepackage{multicol}
\usepackage{color}
\usepackage{colortbl}

\usepackage{showlabels}

\newtheorem{theorem}{Theorem}[section]
\newtheorem{lemma}[theorem]{Lemma}
\newtheorem{proposition}[theorem]{Proposition}
\newtheorem{corollary}[theorem]{Corollary}

\newcommand{\p}{\partial}

\renewcommand{\d}{\delta}
\newcommand{\D}{\Delta}

\renewcommand{\L}{\Lambda}

\begin{document}
\DeclareGraphicsExtensions{.jpg,.pdf,.mps,.png}
\title{Relations between some topological indices and the line graph}

\author[W. Carballosa]{Walter Carballosa}
\address{Department of Mathematics and Statistics, Florida International University, 11200 SW 8th Street
Miami, FL 33199, USA.}
\email{waltercarb@gmail.com}

\author[Ana Granados]{Ana Granados$^{(1)}$}
\address{Saint Louis University, Madrid Campus
Avenida del Valle, 34 - 28003 Madrid, Spain}
\email{aportil2@slu.edu}
\thanks{$^{(1)}$ Supported in part by two grants from Ministerio de Econom{\'\i}a y Competitividad, Agencia Estatal de
Investigación (AEI) and Fondo Europeo de Desarrollo Regional (FEDER) (MTM2016-78227-C2-1-P and MTM2017-90584-REDT), Spain.}

\author[Domingo Pestana]{Domingo Pestana$^{(1)}$}
\address{Departamento de Matem\'aticas, Universidad Carlos III de Madrid,
Avenida de la Universidad 30, 28911 Legan\'es, Madrid, Spain}
\email{dompes@math.uc3m.es}

\author[Ana Portilla]{Ana Portilla$^{(1)}$}
\address{Saint Louis University, Madrid Campus
Avenida del Valle, 34 - 28003 Madrid, Spain}
\email{agranado@slu.edu}

\author[Jos\'e M. Sigarreta]{Jos\'e M. Sigarreta$^{(1)}$}
\address{Facultad de Matem\'aticas, Universidad Aut\'onoma de Guerrero,
Carlos E. Adame No.54 Col. Garita, 39650 Acalpulco Gro., Mexico,
and Instituto de F{\'\i}sica, Benem\'erita Universidad Aut\'onoma de Puebla, Apartado Postal J-48, Puebla 72570, M\'exico}
\email{jsmathguerrero@gmail.com}

\date{\today}

\maketitle{}

\centerline{{\bf RELATIONS BETWEEN SOME TOPOLOGICAL INDICES AND THE LINE GRAPH}}

\begin{abstract}
The concepts of geometric-arithmetic and harmonic indices were introduced in the area of
chemical graph theory recently. They have proven to correlate well with physical and chemical properties of some molecules.
The aim of this paper is to obtain new inequalities involving  the first Zagreb, the harmonic, and the geometric-arithmetic $GA_1$ indices.
Furthermore, inequalities relating these indices and line graphs are proven.
\end{abstract}

{\it Keywords: First Zagreb index, Geometric-arithmetic index, Harmonic index, Vertex-degree-based topological index, Line graph.}

{\it 2010 AMS Subject Classification numbers: 05C07, 92E10.} 

\section{Introduction}

A single number representing a chemical structure in graph-theoretical terms via the
molecular graph, is called a topological descriptor.  If in addition this descriptor correlates with
a molecular property, then it is called a topological index, and it is used to understand physicochemical
properties of chemical compounds.
A reason why these indices are relevant, is that they capture some of the properties of a molecule in a single number.
For this reason, hundreds of topological indices have been introduced and studied, starting with the
seminal work by Wiener in which he modeled physical properties of alkanes by using the sum of all shortest-path distances of
a (molecular) graph (see \cite{Wi}).

Topological indices based on degrees of vertices have been
used for several decades. Among them, a proportion has been recognized as a useful tool in
chemical research. Probably, the most known descriptor is the Randi\'c connectivity
index ($R$) \cite{R}. In fact, there are thousands of papers and a couple of books about this molecular descriptor (see, e.g., \cite{GF}, \cite{LG}, \cite{LS}, \cite{RS}, \cite{RS0} and the references therein).
During many years, scientists were trying to improve the predictive power of this index. As a consequence,  a large number of new topological
descriptors similar to the original Randi\'c index were proposed.
The first and second Zagreb indices,
denoted by $M_1$ and $M_2$, respectively are among the main ones. They are defined as
$$
M_1(G) = \sum_{uv\in E(G)} (d_u + d_v) = \sum_{u\in V(G)} d_u^2,
\qquad
M_2(G) = \sum_{uv\in E(G)} d_u d_v ,
\qquad
$$
where $uv$ denotes the edge of the graph $G$ connecting the vertices $u$ and $v$, and
$d_u$ is the degree of the vertex $u$.
The interest on these indices is shown in works such as \cite{BF}, \cite{D2}, \cite{FGE}, \cite{L}.

Another remarkable topological descriptor is the \emph{harmonic} index, defined in \cite{Faj} as
$$
H(G) = \sum_{uv\in E(G)}\frac{2}{d_u + d_v}\,,
$$
This one did not attract much attention when introduced, but in the last few years, a remarkably large number of papers studying the properties of such index have appeared (see, e.g., \cite{DBAV,FMS,Liu13,RS,WTD,Wu,Wu2,2,Z,Zhong1,Zhong3,ZX,Zhu2}). It also has a chapter in the recent survey \cite{RS5}.

Furthermore, its chemical applicability has recently began to be investigated for example in \cite{Furtula,Gutman8}.
It has been discovered that, regarding physical and chemical properties, it is similar in correlation to the
well-known Randi\'c index. More concretely,
in the paper of Gutman and To\v{s}ovi\'c \cite{Gutman8}, the correlation abilities of $20$ vertex-degree-based topological indices were tested for the case of standard
heats of formation and normal boiling points of octane isomers.
It is remarkable to realize that the harmonic index has good correlation coefficients for this situation.

\smallskip

For a given class of graphs, two questions that arise naturally are, on one hand,
estimating bounds for indices and, on the other, finding the graphs
which are extremal for those indices.
Many results have been
obtained along these lines, in particular for $H(G)$.
For example, Favaron, Mah\'eo and Sacl\'e \cite{FMS} considered the relationship between the harmonic index and the eigenvalues of graphs, and Zhong and Xu \cite{Zhong1,2,Zhong3} determined the minimum and maximum values of the harmonic index for connected graphs, trees, unicyclic graphs, and bicyclic graphs, and characterized the corresponding extremal graphs, respectively.
For example, trees with maximum and minimum harmonic index are the path $P_{n}$ and the star $S_{n}$ respectively.

Another topological descriptor that we shall deal with is the \emph{geometric-arithmetic index} defined in \cite{VF} as
$$
GA_1(G) = GA(G) = \sum_{uv\in E(G)}\frac{\sqrt{d_u d_v}}{\frac12 (d_u + d_v)} \,.
$$
Although $GA$ was introduced barely a decade ago, there are already many papers dealing with this index
(see, e.g., \cite{D}, \cite{DGF}, \cite{DGF}, \cite{DZT}, \cite{MRS}, \cite{PST}, \cite{RS2}, \cite{RS3}, \cite{VF} and the references therein).
The results in \cite[p.598]{DGF} show that the $GA$ index gathers the same information on the molecule under study as other geometric-arithmetic indices.

Even though the number of possible benzenoid hydrocarbons is huge (for example, for benzenoid hydrocarbons with 35 benzene rings it is $5.85\cdot 10^{21}$ \cite{VGJ}) only about 1000 of them are known.
Therefore, modeling their physicochemical properties is important in order
to predict properties of currently unknown species.
The predicting ability of the $GA_1$ index compared with
Randi\'c index is reasonably better (see \cite[Table 1]{DGF}).
The graphic in \cite[Fig.7]{DGF} (from \cite[Table 2]{DGF}, \cite{TRC}) shows
that there exists a good linear correlation between $GA_1$ and the heat of formation of benzenoid hydrocarbons
(the correlation coefficient is equal to $0.972$).
Furthermore, the improvement in
prediction with $GA_1$ index comparing to Randi\'c index in the case of standard
enthalpy of vaporization is more than 9$\%$. That is why one can think that $GA_1$ index
should be considered in the QSPR/QSAR researches.

Line graphs were initially introduced in the papers \cite{W} and \cite{Kr},
but the terminology was used in \cite{HN} for the first time.
The \emph{line graph} $\mathcal{L}(G)$ of $G$ is a graph whose vertices are the edges of $G$, and two vertices are incident if and only if they have a common end vertex in $G$.
For these line graphs, some  topological indices have been considered, for example in \cite{NKMT}, \cite{RLC} and \cite{SX}.

When studying topological indices, it is needed to find bounds on them which involve several parameters.
The aim of this paper is to obtain new inequalities involving the first Zagreb, the harmonic and the geometric-arithmetic  $GA_1$ indices.
Furthermore, inequalities relating these indices and line graphs are proven.

Throughout this work, $G=(V (G),E (G))$ denotes a (non-oriented) finite simple (without multiple edges and loops) graph such that each connected connected component of $G$ has at least an edge.
Given a graph $G$ and $v\in V(G)$, we denote by $N(v)$ the set of \emph{neighbors} of $v$, i.e.,
$N(v)=\{u\in V(G)|\, uv\in E(G) \}$.
We denote by $\D,\d,n,m$ the maximum degree, the minimum degree and the cardinality of the set of vertices and edges of $G$, respectively.
Also, we denote by $\D_{\mathcal{L}(G)},\d_{\mathcal{L}(G)},n_{\mathcal{L}(G)},m_{\mathcal{L}(G)}$ the maximum degree, the minimum degree and the cardinality of the set of vertices and edges of the line graph $\mathcal{L}(G)$ of $G$, respectively.
By $G_1 \approx G_2$, we mean that the graphs $G_1$ and $G_2$ are isomorphic.
We say that a graph $G$ is \emph{non-trivial} if each connected connected component of $G$ has at least two edges.
Since $\mathcal{L}(P_2)$ is a single vertex without edges,
in order to work with line graphs we just consider non-trivial graphs.
Also, we denote by $\D_{\mathcal{L}(G)},\d_{\mathcal{L}(G)},n_{\mathcal{L}(G)},m_{\mathcal{L}(G)}$ the maximum degree, the minimum degree and the cardinality of the set of vertices and edges of the line graph $\mathcal{L}(G)$ of $G$, respectively.
By $G_1 \approx G_2$, we mean that the graphs $G_1$ and $G_2$ are isomorphic.

\section{Fist Zagreb Index and Geometric-Arithmetic Index}

%

First, we obtain an upper bound of the first Zagreb index just in terms on $m$ and $\D$.

\begin{theorem} \label{t:m1}
If $G$ is a graph with $m$ edges and maximum degree $\D$, then
$$
M_1(G) \le \max \big\{2\Delta^2 +m^2 +(6-2\Delta) m - 2\Delta -4\,,\; 2\Delta^2 +m^2 +(4-2\Delta) m + 4 \,,\; m(m-1) \big\} .
$$
\end{theorem}

\begin{proof}
Since $d_u + d_v\le m+1$ for every $uv\in E(G)$, we have
$$
M_1(G) = \sum_{v\in V(G)} d_v^2
= \sum_{uv\in E(G)}(d_u + d_v)
\le m (m+1) ,
$$
If $d_u + d_v = m+1$ for some edge $uv\in E(G)$, then a fortiori one of the vertices, say $u$, must have maximum degree $\Delta$ and so the other one, say $v$, must have degree $m+1-\Delta$. The rest of vertices must have degree $1$ or $2$. Therefore, it is clear that
$$
\begin{aligned}
M_1(G)  & = \sum_{uv\in E(G)}(d_u + d_v) = \sum_{v\in V(G)} d_v^2 \le \Delta^2+ (m+1-\Delta)^2 + 2^2 (\Delta-1+m+1-\Delta-1) \\
& = 2\Delta^2 +m^2 +6m-2\Delta(m+1)-4 = 2\Delta^2 +m^2 +(6-2\Delta) m - 2\Delta -4 \,.
\end{aligned}
$$
If $d_u + d_v = m$ for some edge $uv\in E(G)$ then a fortiori one of the vertices, say $u$, must have maximum degree $\Delta$ and so the other one, say $v$, must have degree $m-\Delta$. Let us observe that in this case we have an extra edge with extremes distinct of $u$ and $v$. Depending on where it is situated this extra edge we have two possible configurations. In the first one all the vertices which are distinct of $u$ and $v$ must have degree $1$ or $2$, but on the other one we have two vertices (which are different of $u$ and$v$) with degree $3$. Therefore, it is clear that
$$
\begin{aligned}
M_1(G)  & = \sum_{uv\in E(G)}(d_u + d_v) = \sum_{v\in V(G)} d_v^2 \le \Delta^2+ (m-\Delta)^2 + + 3^3 \times 2 + 2^2 (\Delta-1+m+1-\Delta-1-2) \\
& = 2\Delta^2 +m^2 -2\Delta m +18 +4(m-4) = 2\Delta^2 +m^2 +(4-2\Delta) m + 4 \,.
\end{aligned}
$$
Finally, if $d_u + d_v = m-1$ for some edge $uv\in E(G)$, then
$$
M_1(G) = \sum_{uv\in E(G)}(d_u + d_v) \le \sum_{uv\in E(G)} (m-1) =m(m-1).
$$
\end{proof}

Theorem \ref{t:m1} and \cite[Corollary 2]{PST}  have the following consequence.

\begin{corollary} \label{t:line5}
If $G$ is a non-trivial graph with $m$ edges and maximum degree $\D$, then
$$
GA_1(G) + GA_1(\mathcal{L}(G))
\le \frac{1}{2}\, \max \big\{2\Delta^2 +m^2 +(6-2\Delta) m - 2\Delta -4\,,\; 2\Delta^2 +m^2 +(4-2\Delta) m + 4 \,,\; m(m-1) \big\} .
$$
\end{corollary}

The Platt number of a graph $G$ is defined (see, e.g., \cite{H}) as
$$
P(G) = \sum_{ uv\in E(G) } (d_{u} + d_{v}-2) \, .
$$

\begin{theorem} \label{t:line12}
If $G$ is a non-trivial graph with $m$ edges, then
$$
m_{\mathcal{L}(G)}
= \frac12 \, P(G) ,
\qquad
P(G)
= M_1(G)-2m ,
\qquad
m
= \frac12 \, \big(M_1(G)-P(G) \big).
$$
\end{theorem}

\begin{proof}
Since the vertex in $\mathcal{L}(G)$ corresponding to $uv\in E(G)$ has degree $d_u + d_v - 2$, we have
$$
P(G)
= \sum_{ uv\in E(G) } (d_{u} + d_{v}-2)
= \sum_{ uv\in V(\mathcal{L}(G)) } d_{uv}
= 2 m_{\mathcal{L}(G))}
= M_1(G)-2m ,
$$
and the three equalities hold.
\end{proof}

Proposition 3 in \cite{PST} has the following consequence:

\begin{corollary} \label{c:line12}
If $G$ is a non-trivial graph with maximum degree $\D$ and minimum degree $\d$, then
$$
\begin{aligned}
\frac{ \sqrt{(\D-1)( \d-1)} }{\D + \d-2} \, P(G)
& \le GA_1(\mathcal{L}(G))
\le \frac12 \, P(G) ,
\\
\frac{ \sqrt{\D \d} }{\D + \d} \, \big(M_1(G)-P(G) \big)
& \le GA_1(G)
\le \frac12 \, \big(M_1(G)-P(G) \big).
\end{aligned}
$$
\end{corollary}

The study of Gromov hyperbolic graphs is a subject of increasing interest, both in pure and applied mathematics (see, e.g., \cite{BRS},
\cite{BKM}, \cite{J}, \cite{MRSV}, \cite{RSVV} and the references cited therein).
We say that a graph $G$ is $t$-\emph{hyperbolic} $(t\ge 0)$ if any side of every geodesic triangle in $G$
is contained in the $t$-neighborhood of the union of the other two sides.
We define the \emph{hyperbolicity constant} $\d(G)$ of a graph $G$ as the infimum of the constants $t\ge 0$ such that $G$ is $t$-hyperbolic.
We consider that every edge has length $1$.

The following inequality relates the geometric-arithmetic index of $\mathcal{L}(G)$ with the hyperbolicity constant of $G$.

\begin{theorem} \label{t:delta}
If $G$ is a non-trivial graph that is not a tree, then
$$
GA_1(\mathcal{L}(G)) \ge \frac{(4\d(G)-1)^{3/2}}{2\d(G)}\, .
$$
\end{theorem}

\begin{proof}
Since $G$ is not a tree, we have $\d(G)>0$.
Also, $\d(G)$ is always an integer multiple of $\frac14$ by \cite[Theorem 2.6]{BRS}
and $\d(G)\notin \{\frac14\,,\frac12\}$ by \cite[Theorem 11]{MRSV}, since $G$ is a simple graph.
Hence, $\d(G)\ge 3/4$.

The function $f(x)=\frac{2(x-1)^{3/2}}{x}$ is increasing in $[1,\infty)$, since
$$
f'(x)= \frac{(x-1)^{1/2}}{x^2} \big(x+2\big) >0
$$
for every $x \in (1,\infty)$.
By \cite[Corollary 4]{PST} we get that
$$
GA_1(\mathcal{L}(G)) \ge \frac{2(m-1)^{3/2}}{m}\, .
$$
Since $\d(G) \le m/4$ by \cite[Corollary 20]{RSVV}, we conclude
$$
GA_1(\mathcal{L}(G)) \ge \frac{2(m-1)^{3/2}}{m} \ge \frac{2(4\d(G)-1)^{3/2}}{4\d(G)}\, ,
$$
since $m \ge 4\d(G) \ge 3$.
\end{proof}

The following result in \cite{gaCadiz} relates the geometric-arithmetic and the first Zagreb indices.

\begin{theorem} \label{t:gam}
If $G$ is a graph with $m$ edges, maximum degree $\D$ and minimum degree $\d$, then
$$
GA_1(G)
\ge \min \Big\{ \, \frac{1}{2 \D} \,,\, \frac{2\sqrt{ \D\d}}{(\D+\d)^2} \, \Big\} M_1(G) ,
$$
and the equality is attained if $G$ is a regular graph.
\end{theorem}

In \cite{PST} appears the following result.

\begin{theorem} \label{t:line10}
If $G$ is a non-trivial graph with $m$ edges, then
$$
M_1(\mathcal{L}(G))
= 4m - 4M_1(G) +2 M_2(G) + F(G) .
$$
\end{theorem}

\begin{proposition} \label{c:gamlinea}
If $G$ is a non-trivial graph with $m$ edges, maximum degree $\D$ and minimum degree $\d$, then
$$
GA_1(\mathcal{L}(G))
\ge \min \Big\{ \, \frac{1}{4(\D-1)} \,,\, \frac{\sqrt{ (\D-1)(\d-1)}}{(\D+\d-2)^2} \, \Big\} \big(4m - 4M_1(G) +2 M_2(G) + F(G)\big) .
$$
\end{proposition}

\begin{proof}
Theorem \ref{t:gam} gives
$$
GA_1(\mathcal{L}(G))
\ge \min \Big\{ \, \frac{1}{2 \D_{\mathcal{L}(G)}} \,,\, \frac{2\sqrt{ \D_{\mathcal{L}(G)}\d_{\mathcal{L}(G)}}}{(\D_{\mathcal{L}(G)}+\d_{\mathcal{L}(G)})^2} \, \Big\} M_1(\mathcal{L}(G)) .
$$
Since
$\D_{\mathcal{L}(G)} \le 2\D-2$ and $\d_{\mathcal{L}(G)} \ge 2\d-2$,
the argument in the proof of Theorem \ref{t:gam} gives, in fact,
$$
GA_1(\mathcal{L}(G))
\ge \min \Big\{ \, \frac{1}{4(\D-1)} \,,\, \frac{\sqrt{ (\D-1)(\d-1)}}{(\D+\d-2)^2} \, \Big\} M_1(\mathcal{L}(G)) .
$$
and Theorem \ref{t:line10} allows to conclude
$$
GA_1(\mathcal{L}(G))
\ge \min \Big\{ \, \frac{1}{4(\D-1)} \,,\, \frac{\sqrt{ (\D-1)(\d-1)}}{(\D+\d-2)^2} \, \Big\} \big(4m - 4M_1(G) +2 M_2(G) + F(G)\big) .
$$
\end{proof}

\section{Harmonic index of line graphs}

The following result appears in \cite[Theorem 2.2]{Liu13}.

\begin{theorem} \label{t:24}
If $G$ is a graph with $n$ vertices, then
$$
H(G) \leq \dfrac{n}{2} \,,
$$
with equality if and only if every connected component of $G$ is regular.
\end{theorem}

Since $n_{\mathcal{L}(G)}=m$,
Theorem \ref{t:24} has the following consequence.

\begin{corollary} \label{c:24}
If $G$ is a non-trivial graph with $m$ edges, then
$$
H(\mathcal{L}(G)) \leq \dfrac{m}{2} \,,
$$
with equality if and only if each connected component of $G$ is regular or biregular.
\end{corollary}

%

We need the following technical result.

\begin{lemma} \label{l:line11}
We have for any positive integers $3 \le k\le \D$ and $1\le x_1,x_2, \dots,x_k \le \D$
$$
\frac{2}{k-1} \sum_{\substack{ 1\le i,j \le k \\ i<j}} \frac1{x_i+x_j+2k-4}
\le \sum_{j=1}^k \frac1{x_j+k}
\le 2\,\frac{\D+2k-3}{k^2-1} \sum_{\substack{ 1\le i,j \le k \\ i<j}} \frac1{x_i+x_j+2k-4} \, .
$$
\end{lemma}

\begin{proof}
Since $3 \le k\le \D$, $f(x)=\frac{x+\D+2k-4}{x+k}$ is a decreasing function if $x \in [1,\infty)$, and
$$
\begin{aligned}
& \qquad \frac{x_j+\D+2k-4}{x_j+k}
\le \frac{\D+2k-3}{k+1} \,,
\\
\frac1{x_j+k}
& \le \frac{\D+2k-3}{k+1} \, \frac1{x_j+\D+2k-4}
\le \frac{\D+2k-3}{k+1} \, \frac1{x_i+x_{j}+2k-4} \,,
\\
\frac1{x_j+k}
& \le \frac{1}{k-1} \sum_{\substack{ 1\le i \le k \\ i\neq j}} \frac{\D+2k-3}{k+1} \, \frac1{x_i+x_{j}+2k-4} \,,
\\
\sum_{j=1}^k \frac1{x_j+k}
& \le \frac{\D+2k-3}{k^2-1} \sum_{\substack{ 1\le i,j \le k \\ i\neq j}} \frac1{x_i+x_{j}+2k-4} \,,
\end{aligned}
$$
and the second inequality holds.

On the other hand, since $x_i \ge 1$ and $k \ge 3$, a similar argument gives
$$
\begin{aligned}
\frac{x_i+x_j +2k-4}{x_j+k}
& \ge \frac{x_j+ 2k-3}{x_j+k} \ge 1,
\\
\frac1{x_j+k}
& \ge \frac1{x_i+x_{j}+2k-4} \,,
\\
\frac1{x_j+k}
& \ge \frac{1}{k-1} \sum_{\substack{ 1\le i \le k \\ i\neq j}} \frac1{x_i+x_{j}+2k-4} \,,
\\
\sum_{j=1}^k \frac1{x_j+k}
& \ge \frac{1}{k -1} \sum_{\substack{ 1\le i,j \le k \\ i\neq j}} \frac1{x_i+x_{j}+2k-4} \,,
\end{aligned}
$$
and the first inequality holds.
\end{proof}

Note that $g(t)=2\frac{\D+2t-3}{t^2-1}$ is a decreasing function when $t \in [3,\D]$ since
$$
g'(t)
= 2 \,\frac{2(t^2-1)-(\D+2t-3)2t}{(t^2-1)^2}
= -4 \,\frac{t^2 + (\D-3)t + 1}{(t^2-1)^2}
<0,
$$
and thus
$$
g(t)
\le g(3)
= \frac{\D+3}{4} \,
$$
for $t \in [3,\D]$.
This fact and Lemma \ref{l:line11} have the following consequence.

\begin{corollary} \label{c:line11}
We have for any positive integers $3 \le k\le \D$ and $1\le x_1,x_2,\dots,x_k \le \D$
$$
\frac2{\D-1} \sum_{\substack{ 1\le i,j \le k \\ i<j}} \frac1{x_i+x_j+2k-4}
\le \sum_{j=1}^k \frac1{x_j+k}
\le \frac{\D+3}{4} \sum_{\substack{ 1\le i,j \le k \\ i<j}} \frac1{x_i+x_j+2k-4} \, .
$$
\end{corollary}

\begin{theorem} \label{t:line11}
If $G$ is a non-trivial graph with maximum degree $\D$, then
$$
\begin{aligned}
\frac{8}{11}\, H(G)
\le H(\mathcal{L}(G))
\le H(G)
& \qquad \text{ if } \, \D<3,
\\
\frac{4}{\D+3}\, H(G)
\le H(\mathcal{L}(G))
\le (\D-1) H(G)
& \qquad \text{ if } \, 3 \le \D \le 4,
\\
\frac{3}{2\D-1}\, H(G)
\le H(\mathcal{L}(G))
\le (\D-1) H(G)
& \qquad \text{ if } \, \D>4.
\end{aligned}
$$
\end{theorem}

\begin{proof}
Since $\D-1$ is an increasing function on $\D$,
$\frac{4}{\D+3}$ and $\frac{3}{\D+1}$ are decreasing functions,
and $\D-1 \ge 1$ since $\D \ge 3$,
by linearity we can assume that $G$ is a connected graph.

If $G$ is a cycle graph, then $\mathcal{L}(G) \approx G$ and $H(\mathcal{L}(G))= H(G)$.

If $G$ is a path graph $P_n$, then $\mathcal{L}(G) \approx P_{n-1}$.
Since
$$
H(P_2) = 1
\quad
\text{ and }
\quad
H(P_n) = 2\,\frac23 + (n-3)\frac{1}2 = \frac{3n-1}{6} \, \quad \text{ if } \, n \ge 3,
$$
$H(P_n)$ is an increasing function on $n$ and $H(\mathcal{L}(G)) \le H(G)$.
Also,
$$
\begin{aligned}
\frac{H(G)}{H(\mathcal{L}(G))}
& = \frac{H(P_n)}{H(P_{n-1})}
\le \max \Big\{ \frac{H(P_3)}{H(P_{2})}\,,\; \max_{k\ge 4}\frac{H(P_k)}{H(P_{k-1})} \, \Big\}
\\
& = \max \Big\{ \frac{4}{3}\,,\; \max_{k\ge 4}\frac{3k-1}{3k-4} \, \Big\}
\le \max \Big\{ \frac{4}{3}\,,\; \frac{11}{8} \, \Big\}
= \frac{11}{8} \,,
\end{aligned}
$$
and $\frac{8}{11} H(G) \le H(\mathcal{L}(G))$.

Since $G$ is a connected graph, if $G$ is not isomorphic to either a path or a cycle graph, then $\D \ge 3$.
Denote by $V_3(G)$ the set of vertices in $u\in V(G)$ with degree $d_u\ge 3$.
If $u\in V_3(G)$, then the edges incident to $u$ correspond to a complete graph $\Gamma_u$ with $d_u$ vertices in $\mathcal{L}(G)$
(note that if $u$ and $u'$ are different vertices in $V_3(G)$, then $E(\Gamma_u) \cap E(\Gamma_{u'}) = \emptyset$).
Recall that if $uv \in E(G) = V(\mathcal{L}(G))$, then $d_{uv}=d_u+d_v-2$.
For each fixed $u\in V_3(G)$, Corollary \ref{c:line11} (with $k = d_u$ and $x_j = d_v$) gives
$$
\sum_{ v\in N(u) } \frac{2}{d_u + d_v}
\le \frac{\D+3}{4} \!\!\!\! \sum_{ u_0v_0\in E(\Gamma_u) } \frac{2}{d_{u_0} + d_{v_0}} \, .
$$
Let us define the sets
$E^1(G)=\{uv \in E(G)\,|\; u\in V_3(G)\}$
and $E^1(\mathcal{L}(G)) = \cup_{u\in V_3(G)} E(\Gamma_u)$.
Since $E(\Gamma_u) \cap E(\Gamma_{u'}) = \emptyset$ for different points $u,u' \in V_3(G)$,
and $E^1(\mathcal{L}(G))= \cup_{u\in V_3(G)} E(\Gamma_u)$, we have
$$
\begin{aligned}
\sum_{ uv\in E^1(G)} \frac{2}{d_u + d_v}
& \le \sum_{u\in V_3(G)}  \sum_{v\in N(u)} \frac{2}{d_u + d_v}
\\
& \le \frac{\D+3}{4} \sum_{u\in V_3(G)} \sum_{ u_0v_0\in E(\Gamma_u) } \frac{2}{d_{u_0} + d_{v_0}}
\\
& = \frac{\D+3}{4} \!\!\!\!\!\! \sum_{ u_0v_0\in E^1(\mathcal{L}(G)) } \frac{2}{d_{u_0} + d_{v_0}} \, .
\end{aligned}
$$
Also, for each fixed $u\in V_3(G)$, Corollary \ref{c:line11} gives
$$
\sum_{ v\in N(u) } \frac{2}{d_u + d_v}
\ge \frac2{\D-1} \!\!\!\! \sum_{ u_0v_0\in E(\Gamma_u) } \frac{2}{d_{u_0} + d_{v_0}} \, .
$$
Since it is possible to have an edge in $uv\in E^1(G)$ with $u,v\in V_3(G)$, we have
$$
\begin{aligned}
\sum_{ uv\in E^1(G)} \frac{2}{d_u + d_v}
& \ge \frac12 \sum_{u\in V_3(G)}  \sum_{v\in N(u)} \frac{2}{d_u + d_v}
\\
& \ge \frac12 \,\frac2{\D-1} \sum_{u\in V_3(G)} \sum_{ u_0v_0\in E(\Gamma_u) } \frac{2}{d_{u_0} + d_{v_0}}
\\
& = \frac1{\D-1} \!\!\!\!\!\! \sum_{ u_0v_0\in E^1(\mathcal{L}(G)) } \frac{2}{d_{u_0} + d_{v_0}} \, .
\end{aligned}
$$

Let $\p V_3(G)$ be the set of vertices in $G$ at distance $1$ from $V_3(G)$.
Consider now the connected components $G_1,\dots,G_r$ of $G \setminus E^1(G)$.
We have $d_u \le 2$ for every $u \in V(G_j)$ and $1 \le j \le r$.
Since $G$ is not isomorphic to a path graph, then $G_j$ is a path graph joining either two vertices in $\p V_3(G)$, or a vertex in $\p V_3(G)$ and a vertex with degree $1$,
for each $1 \le j \le r$.
Denote by $v_j^1, v_j^{2},\dots,v_j^{k_j}$ ($k_j \ge 2$) the vertices in $V(G_j)$ ordered in such a way that
$E(G_j)=\{v_j^1 v_j^{2},v_j^2 v_j^{3},\dots,v_j^{k_j-1} v_j^{k_j}\}$ and $v_jv_j^{k_j} \in E(G)$ for some $v_j \in V_3(G)$.
Furthermore, $d_{v_j^{2}}=\cdots= d_{v_j^{k_j}}=2$.

If $G_j$ is a path graph joining two vertices in $\p V_3(G)$, then $d_{v_j^{1}}=2$ and $v_j'v_j^{1} \in E(G)$ for some $v_j' \in V_3(G)$.
Let $\L_j$ be the set of edges in the associated path graph in $\mathcal{L}(G)$ with the $k_j+1$ vertices
$$
\{v_j'v_j^1,v_j^1 v_j^{2},v_j^2 v_j^{3},\dots,v_j^{k_j-1} v_j^{k_j}, v_j^{k_j}v_j\}.
$$
Thus, $card\,\L_j = k_j=1 + card\,E(G_j)$ and
$$
\begin{aligned}
\sum_{ u_0v_0\in \L_j } \frac{2}{d_{u_0} + d_{v_0}}
& \le 2\,\frac25 + (k_j-2)\frac12
\le \frac85  (k_j-1)\frac12
= \frac85 \sum_{ uv\in E(G_j)} \frac{2}{d_u + d_v} \,,
\\
\sum_{ uv\in E(G_j)} \frac{2}{d_u + d_v}
& = (k_j-1)\frac12
\le \max \Big\{ \frac{\D_{\mathcal{L}(G)}+2}{8}\,,\, 1 \Big\}
\Big( 2\,\frac2{2+\D_{\mathcal{L}(G)}} + (k_j-2)\frac12 \Big)
\\
& \le \max \Big\{ \frac{\D}{4}\,,\, 1 \Big\}
\sum_{ u_0v_0\in \L_j } \frac{2}{d_{u_0} + d_{v_0}} \,.
\end{aligned}
$$

If $G_j$ is a path graph joining a vertex in $\p V_3(G)$ and a vertex with degree $1$, then $d_{v_j^{1}}=1$.
Let $\L_j$ be the set of edges in the associated path graph in $\mathcal{L}(G)$ with the $k_j$ vertices $\{ v_j^1 v_j^{2},v_j^2 v_j^{3},\dots,v_j^{k_j-1} v_j^{k_j}, v_j^{k_j}v_j\}$.
Thus, $card\,\L_j = k_j-1 = card\,E(G_j)$ and
$$
\begin{aligned}
\sum_{ u_0v_0\in \L_j } \frac{2}{d_{u_0} + d_{v_0}}
& \le \frac23 + \frac2{5} + (k_j-3)\frac12
< (k_j-2)\frac12 + \frac23
= \sum_{ uv\in E(G_j)} \frac{2}{d_u + d_v} \,,
\\
\sum_{ uv\in E(G_j)} \frac{2}{d_u + d_v}
& = (k_j-2)\frac12 + \frac23
\le \frac{\D_{\mathcal{L}(G)}+2}{4} \Big( \frac23 + \frac2{2+\D_{\mathcal{L}(G)}} + (k_j-3)\frac12 \Big)
\\
& \le \frac{\D}{2} \sum_{ u_0v_0\in \L_j } \frac{2}{d_{u_0} + d_{v_0}}
\,,
\end{aligned}
$$
if $k_j \ge 3$.
If $k_j =2$, then
$$
\begin{aligned}
\sum_{ u_0v_0\in \L_j } \frac{2}{d_{u_0} + d_{v_0}}
& \le \frac2{3 +1}
< \frac2{2 +1}
= \sum_{ uv\in E(G_j)} \frac{2}{d_u + d_v} \,,
\\
\sum_{ uv\in E(G_j)} \frac{2}{d_u + d_v}
& = \frac2{2 +1}
= \frac{\D_{\mathcal{L}(G)} +1}3 \,\frac2{\D_{\mathcal{L}(G)} +1}
\le \frac{2\D -1}3 \sum_{ u_0v_0\in \L_j } \frac{2}{d_{u_0} + d_{v_0}}
\,.
\end{aligned}
$$

Summing up on the edges when $\D \ge 3$, we obtain
$$
H(\mathcal{L}(G))
\le \max \Big\{ \D-1, \,\frac{8}{5} \,, 1 \Big\} H(G)
= (\D-1) H(G),
$$
and
$$
\begin{aligned}
H(\mathcal{L}(G))
& \ge \min \Big\{ \frac4{\D +3}\,, \,\frac{4}{\D} \,, \,\frac 2{\D} \,, \,\frac3{2\D -1} \,\Big\} H(G)
\\
& = \min \Big\{ \frac4{\D +3}\,, \,\frac3{2\D -1} \,\Big\} H(G).
\end{aligned}
$$
This finishes the proof since $4/(\D+3)\le 3/(2\D-1)$ if $3\le\D\le 4$ but if $\D>4$ then we have the opposite inequality.
\end{proof}

\end{document}